\documentclass[12pt]{amsart}
\usepackage{graphicx}
\usepackage{amsmath}
\usepackage{amssymb} 
\usepackage[usenames,dvipsnames]{color}
\usepackage{pinlabel}

\newtheorem{thm}{Theorem}[section]  
\newtheorem{cor}[thm]{Corollary}

\newtheorem{defin}[thm]{Definition} 
\newtheorem{lemma}[thm]{Lemma} 
 
\newtheorem{prop}[thm]{Proposition}

\newcommand{\Emb}{\operatorname{Emb}}
\newcommand{\Unk}{\operatorname{Unk}}
\newcommand{\Diff}{\operatorname{Diff}}
\newcommand{\GL}{\operatorname{GL}}
\newcommand{\Ooo}{\operatorname{O}}
\renewcommand{\H}{\operatorname{{\mathcal H}}}

\newcommand{\I}{\operatorname{I}}
\newcommand{\T}{\operatorname{T}}

\newcommand{\mbbR}{\mathbb{R}}

\newcommand{\aaa}{\mbox{$\alpha$}}
\newcommand{\anch}{\mbox{$\mathfrak{A}$}}
\newcommand{\AF}{\mbox{$\mathfrak{AF}$}}
\newcommand{\bbb}{\mbox{$\beta$}}

\newcommand{\free}{\mbox{$\mathfrak{F}$}} 

\newcommand{\kkk}{\mbox{$\kappa$}}
\renewcommand{\lll}{\mbox{$\lambda$}} 

\newcommand{\ooo}{\mbox{$\omega$}}

\newcommand{\Sss}{\mbox{$\Sigma$}}

\newcommand{\bdd}{\mbox{$\partial$}}
\newcommand{\inter}{\mbox{${\rm int}$}}
\def\zed{{\mathbb Z}}

\begin{document}  

\title[Genus $g+1$ Goeritz group]{Generating the genus $g+1$ Goeritz group of a genus $g$ handlebody}   

\author{Martin Scharlemann}
\address{\hskip-\parindent
        Martin Scharlemann\\
        Mathematics Department\\
        University of California\\
        Santa Barbara, CA USA}
\email{mgscharl@math.ucsb.edu}

\thanks{Research partially supported by National Science Foundation grants.}

\date{\today}

\begin{abstract}  A specific set of $4g+1$ elements is shown to generate the Goeritz group of the genus $g+1$ Heegaard splitting of a genus $g$ handlebody.  These generators are consistent with Powell's proposed generating set for the Goeritz group of the genus $g+1$ splitting of $S^3$.  There are two proofs: one using purely classical techniques and one using thin position.  
\end{abstract}

\maketitle

\section{Introduction}

Following early work of Goeritz \cite{Go}, the {\em genus $g$ Goeritz group} of the $3$-sphere can be described as the isotopy classes of orientation-preserving homeomorphisms of the $3$-sphere that leave the genus $g$ Heegaard splitting invariant.  Goeritz identified a finite set of generators
for the genus $2$ Goeritz group; that work has been recently updated, extended and completed, to give a full picture of the group (see \cite{Sc}, \cite{Ak}, \cite{Cho}).  Goeritz' set of generators was
extended by Powell \cite{Po} to a set of generators for all higher genus Goeritz groups, but his proof that the generators suffice contained a gap \cite{Sc}.  The finite set of elements that Powell proposed
as generators for the full Goeritz group remains a very plausible set, though a proof remains elusive.

One intriguing aspect of the problem is that Gabai's powerful technique of thin position \cite{Ga} is available for objects in $S^3$, such as Heegaard splitting surfaces (see \cite{ST2}), but the technique was not known to Powell.  In addition, one can imagine structuring a proof by induction on the ``co-genus" $k$ of Heegaard splittings of a handlebody: any genus $g$ Heegaard splitting of a genus $g-k$ handlebody $H$ gives rise to a genus $g$ splitting of $S^3$, by identifying $H$ with one of the handlebodies in the standard genus $g - k$ splitting of $S^3$.  In that context, Powell's conjecture would suggest a natural set of generators for the genus $g$ Goeritz group of a genus $g-k$ handlebody (see Section \ref{sect:toGoeritz} for the definition).  As $k$ ascends we eventually have a set of generators for the genus $g$ Goeritz group of the genus $0$ handlebody $B^3$ (or, equivalently, $S^3$).  With that strategy in mind, here we verify Powell's conjecture for the first and easiest case, namely co-genus $1$.  Rephrasing slightly, we exhibit, for any genus $g$ handlebody $H$, a certain finite set of elements that generates the genus $g+1$ Goeritz group $G(H, \Sigma)$ of $H$.  Combining the results of Theorems \ref{thm:main1} and \ref{thm:goeritz}, to which we refer for notation, we show:

\begin{thm} The Goeritz group $G(H, \Sigma)$ of the genus $g$ handlebody $H$ is generated by $4g+1$ elements, namely $2g$ generators of the subgroup $\anch_{\{E_1, ..., E_g\}}$ and $2g+1$ generators of the subgroup $\free_{E_0}$.  
\end{thm}

We will give two proofs that these generators suffice: the first is along classical lines (i. e. without thin position) and the second uses thin position.  Both arguments are given in a slightly different setting -- the isotopies are of an unknotted arc in the handlebody, rather than a Heegaard surface -- but the connection between the two is explained in Section \ref{sect:toGoeritz}.

\section{Embedding an unknotted arc in a ball}

For $M, N$ smooth manifolds, let $\Emb(M, N)$ denote the space of smooth proper embeddings of $M$ into $N$.  Let $\Emb_0(I, B^3) \subset \Emb(I, B^3)$ denote the path-component consisting of those embeddings for which the image is an unknotted arc.  There is a natural fibration $\Emb_0(I, B^3) \to \Emb(\bdd I, \bdd B^3)$ whose fiber is $\Emb_0(I, B^3\; rel\; \bdd I)$ \cite{Pa}.   Following Hatcher's proof of the Smale conjecture, this fiber is contractible \cite[Appendix (6)]{Ha2}, so in particular $\pi_1(\Emb_0(I, B^3)) \cong \pi_1(\Emb(\bdd I, \bdd B^3))$.  The space $\Emb(\bdd I, \bdd B^3)$ is the configuration space $F_2(S^2)$ of ordered pairs of points in the sphere; its fundamental group is the pure braid group of two points in $S^2$, which is trivial.  Hence $\Emb_0(I, B^3)$ is simply connected.  By taking each element of $\Emb_0(I, B^3)$ to its image in $B^3$ we get a natural map $\Emb_0(I, B^3) \to \Unk(I, B^3)$, the space of unknotted arcs in $B^3$; its fiber is the space of automorphisms of the interval $I$, which consists of two contractible components, representing orientation preserving and orientation reversing automorphisms of the interval.  Combining these two observations we discover that the natural map from $\Unk(I, B^3)$ to the configuration space $C_2(S^2)$  of unordered pairs of points in $S^2$ induces an isomorphism between the respective fundamental groups.  Note that $\pi_1(C_2(S^2))$ is commonly called the full braid group $B_2(S^2)$.  We conclude that $\pi_1(\Unk(I, B^3)) \cong B_2(S^2) \cong \mathbb{Z}_2$, \cite[Theorem 1.11]{Bi}.

Now suppose $P$ is a connected planar surface in $\bdd(B^3)$ and $\Unk_P(I, B^3)$ is the space of all unknotted arcs in $B^3$ whose end points lie in $P$.  Exactly the same argument as above shows that $\pi_1(\Unk_P(I, B^3)) \cong B_2(P)$, where the latter is the full braid group of two points in $P$.  It is straightforward to identify a set of generators for $B_2(P)$.  Begin with the ordered configuration space $C_2(P)$ and project to the first point $x_0 \in P$.  The map is a fibration $p: C_2(P) \to C_1(P) = P$ whose fiber is $C_1(P - \{point\}) = P - \{point\}$ \cite{FN}.    Since $P$ is connected and $\pi_2(P)$ is trivial, it follows that $\pi_1(C_2(P))$ is an extension of $\pi_1(P)$ by $\pi_1(P - \{point\})$ and each of these groups admits a well-known collection of generators, one for each boundary component of $P$.  Namely, for each boundary component choose a loop from the base point that is parallel in $P - \{point\}$ to that component.  One of these generators is redundant in $\pi_1(P)$; all are needed in $\pi_1(P - \{point\})$.  To complete this set of generators to a set of generators for $B_2(P)$, add an isotopy of the pair of points that interchanges the pair.  

These rather abstract descriptions translate to this concrete description of a set of generators for $\pi_1(\Unk_P(B^3)) \cong B_2(P)$:  Let $\alpha$ be a short arc in $P$; its endpoints $x_0, x_1$ will be the pair of points whose motion we are describing.  Half rotation of $\alpha$ around its center, exchanging its ends is one generator for $B_2(P)$; call it the {\it rotor} $\rho_0$.   Let $c_1, ..., c_p$ be the boundary components of $P$ and for each $c_i$ choose a loop $\gamma_i$  in $P$ that passes through $x_1$ and is parallel in $P - x_0$ to $c_i$.  Choose these loops so that they intersect each other or  $\aaa$ only in the point $x_1$ (see Figure \ref{fig:AgenP}).  For each $1 \leq i \leq p$ let $\rho_i$ be an isotopy that moves the entire arc $\alpha$ through a loop in $P$ parallel to $\gamma_i$ and back to itself.  This defines each $\rho_i$ up to multiples of the rotor $\rho_0$, so the subgroup $\free$ of $B_2(P)$ defined as that generated by $\rho_i, 0 \leq i \leq p$ is in fact well-defined.  Call $\free$ the {\it freewheeling} subgroup.  $\free$ is an extension of $\pi_1(P)$ by $\mathbb{Z}$.  

A second subgroup $\anch \subset \pi_1(\Unk_P(B^3)) \cong B_2(P)$, called the {\it anchored} subgroup, is defined as those elements which keep the ``anchor" end $x_0$ of $\aaa$ fixed as the other end follows a closed path in $P - \{x_0\}$ that begins and ends at $x_1$.  It corresponds to the fundamental group of the fiber $P - \{point\}$ in the above fibration. More concretely, for each 
$1 \leq i \leq p$ let $\mathfrak{a}_i$ denote the element determined by keeping $x_0$ fixed and 
moving $x_1$ around the loop $\gamma_i$.  The subgroup $\anch$ is generated by the $\mathfrak{a}_i$; it includes any even power of the rotor $
\rho_0$,  via the relation (in Figure \ref{fig:AgenP}) $\mathfrak{a_1}\mathfrak{a_2}...\mathfrak{a_p} = \rho^2$.  

The fibration above shows that together $\anch$ and $\free$ generate the group $\pi_1(\Unk_P(B^3))$, so $\{ \rho_i, 0 \leq i \leq p \}$ and $\{ \mathfrak{a}_i, 1 \leq i \leq p \}$ together constitute a set of generators for $\pi_1(\Unk_P(B^3)).$

 \begin{figure}[ht!]
 \labellist
\small\hair 2pt
\pinlabel $c_1$ at 185 125
\pinlabel $c_2$ at 385 130
\pinlabel $c_i$ at 335 43
\pinlabel $\gamma_1$ at 263 125
\pinlabel $\gamma_2$ at 300 125
\pinlabel $\gamma_i$ at 270 40
\pinlabel $\aaa$ at 190 93
\pinlabel $x_0$ at 140 90
\pinlabel $x_1$ at 305 85
\pinlabel $P$ at 280 160
 \endlabellist
    \centering
    \includegraphics[scale=0.7]{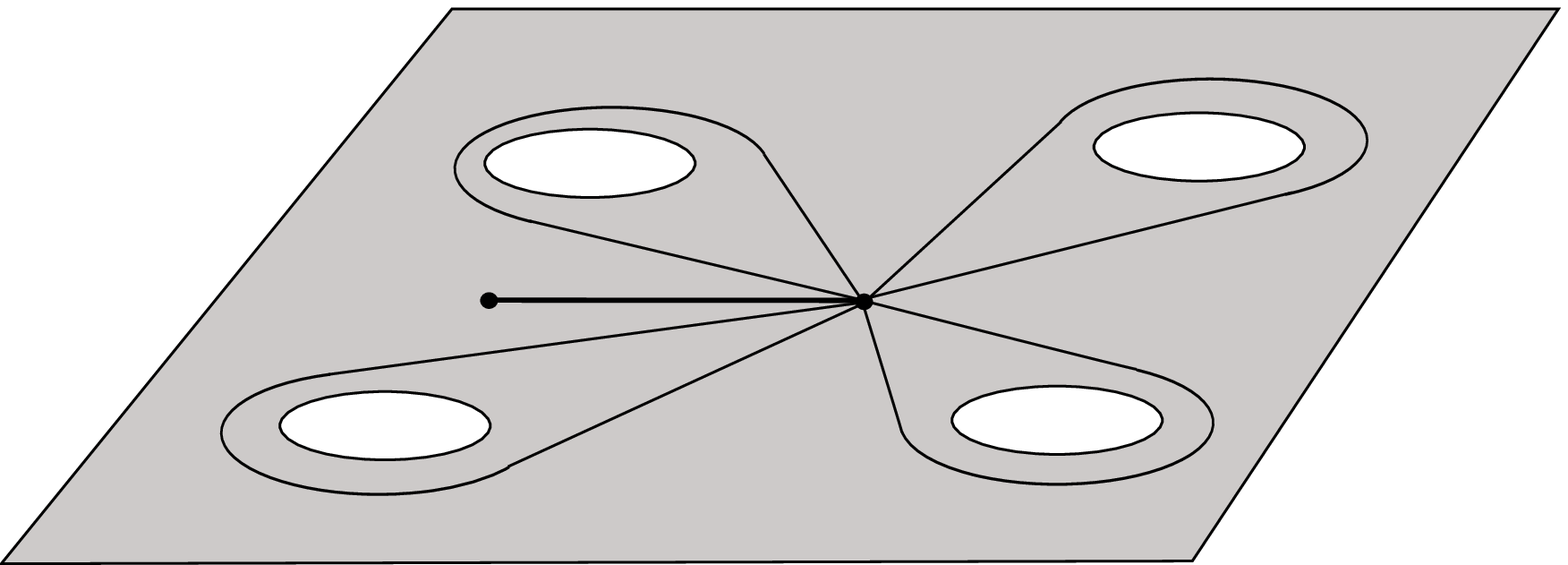}
    \caption{Generating $\anch \subset B_2(P)$} \label{fig:AgenP}
    \end{figure}

\section{Unknotted arcs in a handlebody} \label{sect:unkinH}

The goal of this section is to extend this analysis to describe, for a genus $g$ handlebody $H$, a fairly natural set of generators for the fundamental group of the space $\Unk(I, H)$ of unknotted properly embedded arcs in $H$.  

We begin with a basepoint for $\Unk(I, H)$, i. e. a fixed choice of unknotted arc in $H$.  This is facilitated by viewing $H$ as the product of a planar surface with $I$:  Let $Q$ be a disk $D$ from which $g$ disks $D_1, ..., D_g$ have been removed.  Picture the disks $D_i$ as laid out in a horizontal row in $D$, with a vertical arc $\bbb_i, 1 \leq i \leq g$ descending from each $\bdd D_i \subset \bdd Q$ to $\bdd D$. Further choose a point $x \in inter(Q) - \cup_i \bbb_i$ to the left of the disks $D_i$ and connect it to $\bdd D$ by a horizontal arc $\bbb_0$.  See Figure \ref{fig:QxI}.  Then $Q \times I$ is a handlebody in which $x \times I$ is an unknotted arc $I_0$ in $H$ with end points $x_i = x \times \{i\}, i = 0, 1$.  Furthermore, the $g$ disks $E_i = \bbb_i \times I \subset H, i = 1, ..., g$ constitute a complete collection of meridian disks for $H$.  That is, the complement in $H$ of a regular neighborhood $\eta(\cup_{i=1}^g E_i)$ of $\cup_{i=1}^g E_i$ is a $3$-ball $B^3$ which intersects $\bdd H$ in a planar surface $P$.  The boundary of $P$ has $2g$ components, two copies of each $\bdd E_i, i = 1, ..., g$.

 \begin{figure}[ht!]
 \labellist
\small\hair 2pt
\pinlabel $D_1$ at 90 145
\pinlabel $D_2$ at 190 145
\pinlabel $D_g$ at 290 145
\pinlabel $\beta_1$ at 110 100
\pinlabel $\beta_2$ at 210 100
\pinlabel $\beta_g$ at 300 100
\pinlabel $\beta_0$ at 15 155
\pinlabel $I_0$ at 45 110
\pinlabel $x_0$ at 45 70
\pinlabel $x_1$ at 45 145
\pinlabel $E_0$ at 10 90
\pinlabel $E_1$ at 95 50
\pinlabel $E_2$ at 200 40
\pinlabel $E_g$ at 285 50
\pinlabel $Q$ at 160 190
 \endlabellist
    \centering
    \includegraphics[scale=0.7]{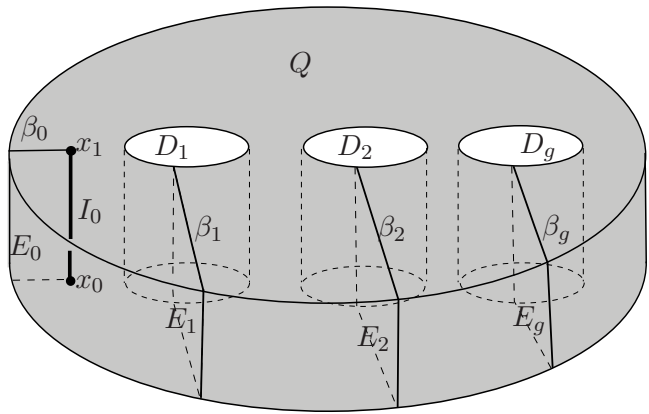}
    \caption{A genus $g$ handlebody $H$ with trivial arc $I_0$} \label{fig:QxI}
    \end{figure}

The disk $E_0 = \bbb_0 \times I$ defines a parallelism between the arc $I_0$ and an arc $\aaa \subset \bdd H$.  Such a disk will be called a {\em parallelism disk} for $I_0$ and the subarc of its boundary that lies on $\bdd H$ will be called a {\em parallel arc} for $I_0$.  It will be convenient, when considering a pair $E_0, E_0'$ of parallelism disks and corresponding parallel arcs $\aaa_0 = E_0 \cap \bdd H, \aaa_0' = E_0' \cap \bdd H$ for $I_0$, to isotope the disks so that they are transverse except where they coincide along $I_0$, and so that they have disjoint interiors near $I_0$.  (This is done by unwinding $E_0'$ along $I_0$).    A standard innermost disk argument shows that the simple closed curves in $E_0 \cap E_0'$ can be removed by an isotopy that does not move $I_0$, after which what remains of $E_0 \cap E_0'$ is their common boundary arc $I_0$ together with a collection of interior arcs whose endpoints are the points of $\aaa_0 \cap \aaa_0'$.  Call this a {\em normal position} for two parallelism disks.

Motivated by the discussion above, we note some obvious elements and subgroups of $\pi_1(\Unk(I, H))$:  Since the pair $(B^3, P)$ is a subset of $(H, \bdd H)$ there is a natural inclusion-induced homomorphism  $\pi_1(\Unk_P(I, B^3)) \to \pi_1(\Unk(I, H)$.  For example, a natural picture of the rotor $\rho_0$ in $\pi_1(\Unk(I, H))$ is obtained by doing a half-twist of $I_0$ in a $3$-ball neighborhood of the disk $\bbb_0 \times I$.  This is shown on the left in Figure \ref{fig:rotorfrak}.  The image of the anchored subgroup $\anch \subset B_2(P)$ in  $\pi_1(\Unk(I, H))$ can be defined much like the anchored subgroup in $B_2(P)$ itself:  hold the end of $I_0$ at $x_0$ fixed while isotoping the end at $x_1$ so that the whole arc $I$ moves around and back to its original position, never letting the moving $I$ intersect any of  the g disks $E_i$.  We denote this subgroup $\anch_{\{E_1, ..., E_g\}} \subset  \pi_1(\Unk(I, H))$.  Two of its generators are shown center and right in Figure \ref{fig:rotorfrak}. There is also a naturally defined {\em freewheeling} subgroup $\free_{E_0} \subset \pi_1(\Unk(I, H))$ consisting of those elements represented by a proper isotopy of the disk $E_0$ through $H$ and back to itself (though perhaps with orientation reversed).  Thus again the rotor $\rho_0$ lies in $\free_{E_0}$, and the kernel of $\free_{E_0} \to \pi_1(\bdd H)$ is generated by the rotor $\rho_0$.  Since $\pi_1(\bdd H)$ is itself generated by $2g$ elements (essentially given by the choice of $\{ E_1, ..., E_g \}$), $\free_{E_0}$ is generated by $2g+1$ elements.

 \begin{figure}[ht!]
 \labellist
\small\hair 2pt
\pinlabel $\rho_0$ at 23 100
\pinlabel $\mathfrak{a}_1$ at 165 90
\pinlabel $\mathfrak{a}_1'$ at 360 135
 \endlabellist
    \centering
    \includegraphics[scale=0.7]{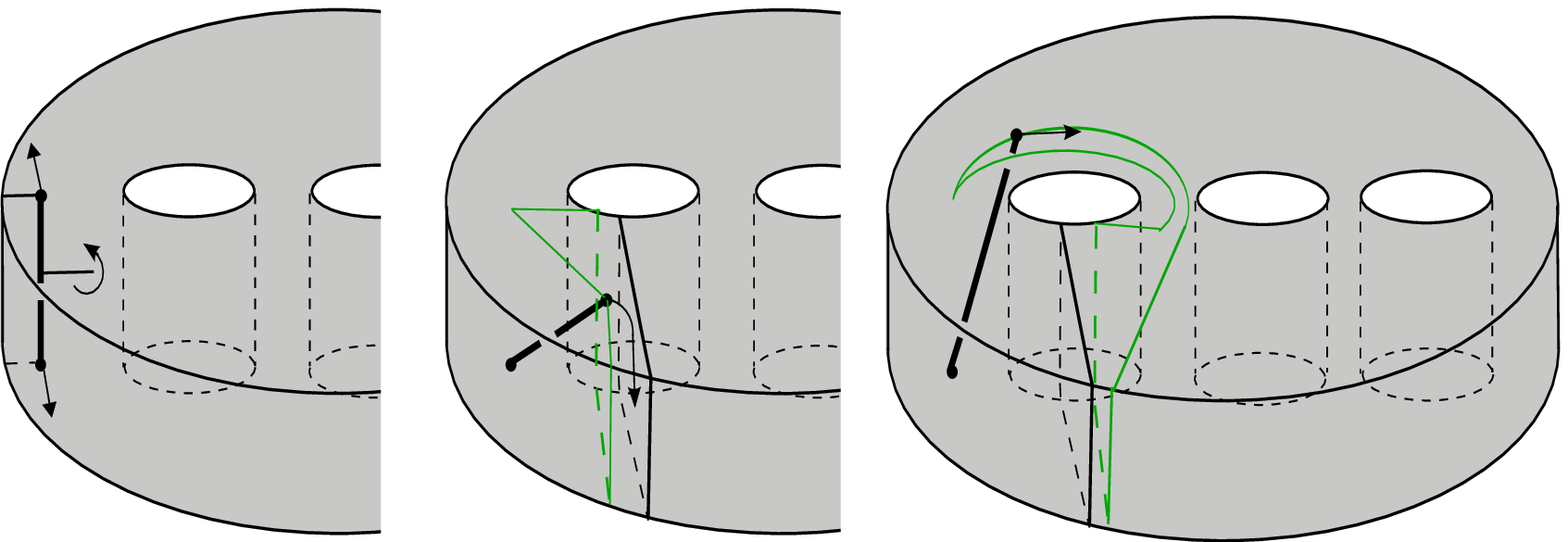}
\caption{The rotor $\rho_0$ and generators $\mathfrak{a}_1, \mathfrak{a}_1' \in \anch$.} \label{fig:rotorfrak}
    \end{figure}

Suppose $\ooo \in \pi_1(\Unk(I, H))$ is represented by a proper isotopy $f_t: I \to H$.  The isotopy extends to an ambient isotopy of $H$ which we continue to denote $f_t$; let $\aaa_{\ooo} \subset \bdd H$ denote $f_1(\aaa_0)$.  Since $f_1(I_0) = I_0$, $f_1(E_0)$ is a new parallelism disk, and $\aaa_{\ooo}$ is the corresponding parallel arc for $I_0$ in $\bdd H$.

\begin{lemma}  \label{lemma:diskback} The parallel arcs $\aaa_0$ and $\aaa_{\ooo}$ for $I_0$ are isotopic rel end points in $\bdd H$ if and only if $\ooo \in \free_{E_0}$.
\end{lemma}

\begin{proof}  If $\ooo \in \free_{E_0}$ then by definition $\aaa_{\ooo} = \aaa_0$.  On the other hand, if $\aaa_{\ooo}$ is isotopic to $\aaa_0$ rel end points, then we may as well assume $\aaa_{\ooo} = \aaa_0$, for the isotopy from $\aaa_{\ooo}$ to $\aaa_0$ doesn't move $I_0$.  Then, thickening $I_0$ slightly to $\eta(I_0)$, $E_0$ and $f_1(E_0)$ are properly embedded disks in the handlebody $H - \eta(I_0)$ and have the same boundary.  A standard innermost disk argument shows that then $f_1(E_0)$ may be isotoped rel $\bdd$ (so, in particular, the isotopy leaves $I_0$ fixed) until $f_1(E_0)$ coincides with $E_0$, revealing that $\ooo \in \free$. 
\end{proof}  

A sequence of further lemmas will show:

\begin{thm}  \label{thm:main1} The subgroups $\anch_{\{E_1, ..., E_g\}}$ and $\free_{E_0}$ together generate all of $\pi_1(\Unk(I, H))$, so the union of their generators is an explicit set of generators for $\pi_1(\Unk(I, H))$.
\end{thm}

What is perhaps surprising about this theorem is that the subgroups themselves depend heavily on our choice of the disks $\{E_0, ..., E_g\}$.  Recognizing this dependence, let the combined symbol $\AF_{\{E_0, ..., E_g\}}$ denote the subgroup of $\pi_1(\Unk(I, H))$ generated by $\anch_{\{E_1, ..., E_g\}}$ and $\free_{E_0}$.  

\begin{lemma}   \label{lemma:drop0} Suppose $E_0' \subset H$ is another parallelism disk that lies entirely in $H - \{E_1, ..., E_g\}$.  Then $\AF_{\{E_0', E_1, ..., E_g\}} = \AF_{\{E_0, E_1, ..., E_g\}}$.
\end{lemma}

\begin{proof}  Put the pair $E_0, E_0'$ in normal position. The proof is by induction on the number $|E_0 \cap E_0'|$ of arcs in which their interiors intersect. 

If $|E_0 \cap E_0'| = 0$, so the disks intersect only in $I_0$, then let $E_{\cup} \subset H - \{E_1, ..., E_g\}$ be the properly embedded disk that is their union.  Rotating one end of $I_0$ fully around a slightly pushed-off copy of $E_{\cup}$ describes an element $\mathfrak{a} \in \anch_{\{E_1, ..., E_g\}}$ for which a representative isotopy carries the disk $E_0$ to $E_0'$. See Figure \ref{fig:Ecup}.  In particular, if $f$ is any element of $\free_{E_0}$ then the product $\mathfrak{a}^{-1} f \mathfrak{a}$ has a representative isotopy which carries $E_0'$ to itself.  Hence
$\mathfrak{a}^{-1} f \mathfrak{a} \in \free_{E_0'}$ so $f \in \AF_{\{E_0', E_1,..., E_g\}}$.  Thus 
$\free_{E_0} \subset \AF_{\{E_0', E_1, ..., E_g\}}$, so $\AF_{\{E_0, E_1,..., E_g\}}\subset \AF_{\{E_0', E_1, ..., E_g\}}$. The symmetric argument shows that $\AF_{\{E_0', E_1, ..., E_g\}} \subset \AF_{\{E_0, E_1,..., E_g\}}$ and so $\AF_{\{E_0, E_1,..., E_g\}} = \AF_{\{E_0', E_1, ..., E_g\}}$ in this case.

 \begin{figure}[ht!]
 \labellist
\small\hair 2pt
\pinlabel $E_{\cup}$ at 38 165
\pinlabel $E_0$ at 10 95
\pinlabel $E_0'$ at 50 120
\pinlabel $\mathfrak{a}$ at 220 150
\pinlabel $\mathfrak{a}$ at 405 150
\pinlabel $\mathfrak{a}$ at 620 160
 \endlabellist
    \centering
    \includegraphics[scale=0.5]{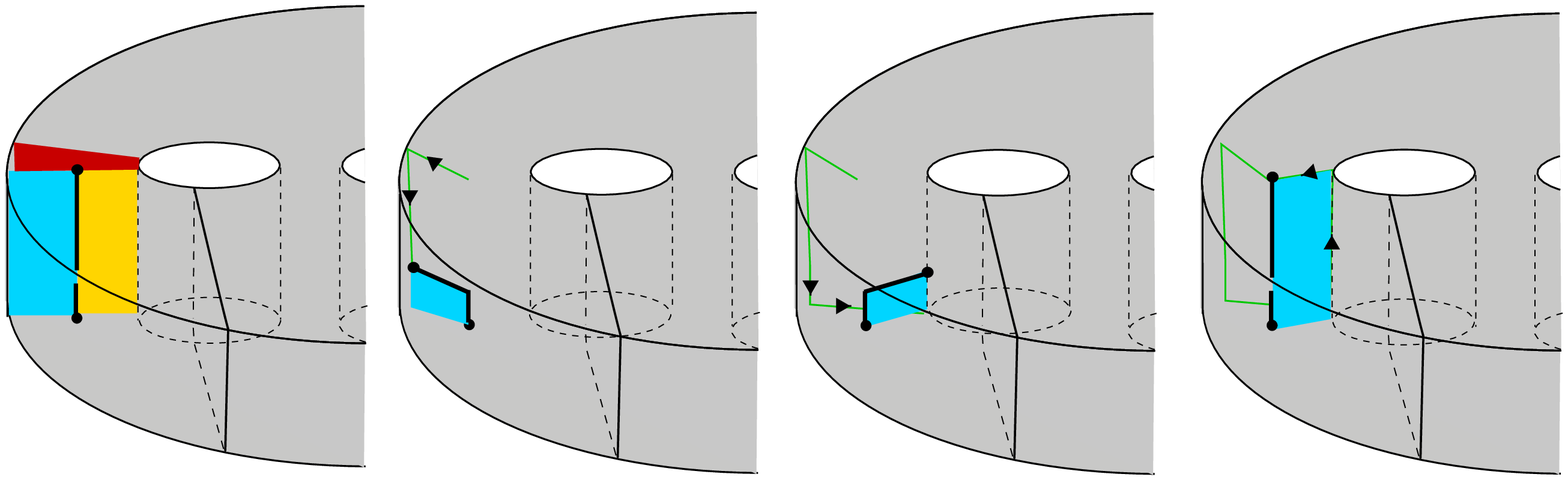}
\caption{$\mathfrak{a} \in \anch$: circling around $E_{\cup}$ brings $E_0$ to $E_0'$.} \label{fig:Ecup}
    \end{figure}

The argument just given shows that for any pair of parallelism disks $F_0, F_0' \subset H - \{E_1, ..., E_g\}$ which intersect only along $I_0$, $\AF_{\{F_0', E_1, ..., E_g\}} = \AF_{\{F_0, E_1, ..., E_g\}}$.  Suppose inductively that this is true whenever $|F_0 \cap F_0'| \leq k$ and, for the inductive step, suppose that $|E_0 \cap E_0'| = k+1.$  Among all arcs of $E_0 \cap E_0'$, let $\bbb$ be outermost in $E_0'$, so that the subdisk $E^* \subset E_0'$ cut off by $\bbb$ does not contain $I_0$ in its boundary, nor any other point of $E_0$ in its interior.   Then attaching $E^*$ along $\bbb$ to the component of $E_0 - \bbb$ that contains $I_0$ gives a parallelism disk $F_0$ that is disjoint from $E_0$ and intersects $E_0'$ in $\leq k$ arcs.  It follows by inductive assumption that $\AF_{\{E_0,E_1, ..., E_g\}} = \AF_{\{F_0, E_1..., E_g\}} = \AF_{\{E_0',E_1, ..., E_g\}}$ as required.
\end{proof}

Following Lemma \ref{lemma:drop0} there is no loss in dropping $E_0$ from the notation, so $\AF_{\{E_0, E_1, ..., E_g\}}$ will henceforth be denoted simply $\AF_{\{E_1, ..., E_g\}}$.

\begin{lemma}   \label{lemma:drop1} Suppose $E_* \subset H$ is a disk in $H - (I_0 \cup E_1 \cup ... \cup E_g\}$, so that $\{E_*, E_2, ..., E_g\}$ is a complete set of meridian disks for $H$.  Then $\AF_{\{E_1, E_2, ..., E_g\}} = \AF_{\{E_*, E_2, ..., E_g\}}$.  That is, the subgroup $\AF_{\{-,  E_2, ..., E_g\}}$ is the same, whether we fill in $E_1$ or $E_*$.  
\end{lemma}

\begin{proof}  Since all $g+1$ meridian disks $\{E_*, E_1, E_2, ..., E_g\}$ are mutually disjoint and all are disjoint from $I_0$, both $\AF_{\{E_1, E_2, ..., E_g\}}$ and $\AF_{\{E_*, E_2, ..., E_g\}}$ can be defined using a parallelism disk $E_0$ that is disjoint from all of the disks $\{E_*, E_1, E_2, ..., E_g\}$.  It follows that $\free_{E_0}$ is a subgroup of both $\AF_{\{E_1, E_2, ..., E_g\}}$ and $\AF_{\{E_*, E_2, ..., E_g\}}$.  Hence it suffices to show that $\anch_{\{E_*, E_2, ..., E_g\}} \subset \AF_{\{E_1, E_2, ..., E_g\}}$ and $\anch_{\{E_1, E_2, ..., E_g\}} \subset \AF_{\{E_*, E_2, ..., E_g\}}$.  We will prove the latter; the former follows by a symmetric argument.

 \begin{figure}[ht!]
 \labellist
\small\hair 2pt
\pinlabel $E_*$ at 145 150
\pinlabel $E_0$ at 198 175
\pinlabel ${E_*}'$ at 140 175
\pinlabel $c$ at 73 167
\pinlabel $E_1$ at 98 50
\pinlabel $E_i$ at 288 50
 \endlabellist
    \centering
    \includegraphics[scale=0.8]{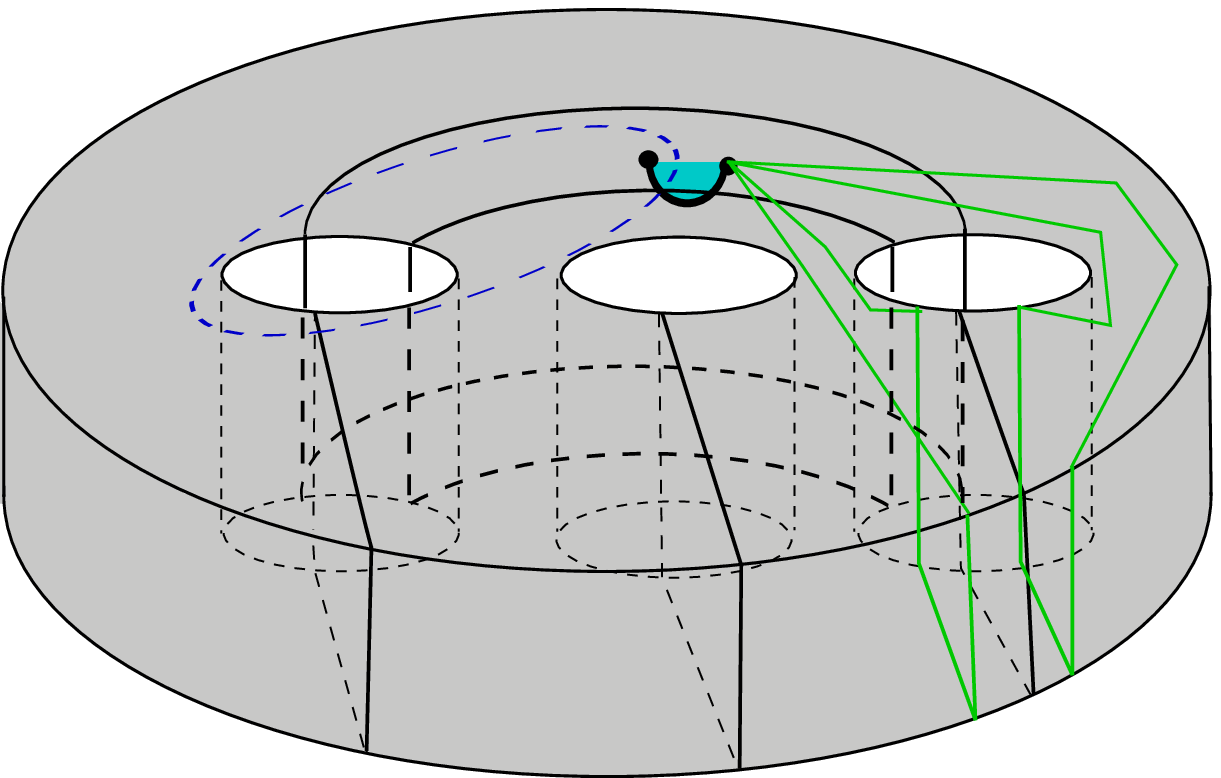}
\caption{} \label{fig:drop1}
    \end{figure}

Extend a regular neighborhood of  $\cup_{i=2}^g E_i$ to a regular neighborhood $Y$ of $\cup_{i=1}^g E_i$ and $Y^*$ of  $E_* \cup \cup_{i=2}^g E_i$.  The disk $E_*$ is necessarily separating in the ball $B^3 = H - Y$ and, since $H - Y^*$ is also a ball, it follows that the two sides of $E_1$ in $\bdd B^3$ lie in different components of $B^3 - E_*$.  Put another way, there is a simple closed curve $c$ in $\bdd H$ which is disjoint from $\{E_2, E_3, ..., E_g\}$ but intersects each of $E_1$ and $E_*$ in a single point.  Let $f \in \free_{E_0}$ be the element represented by isotoping $E_0$ around the circle $c$ in the direction so that it first passes through $E_1$ and then through $E_*$.  

The image of $E_*$, after the isotopy $f$ is extended to $H$, is a disk ${E_*}'$ that is isotopic to $E_*$ in $B^3$ but not in $B^3 - I_0$.  The isotopy need not disturb the disks $\{E_2, E_3, ..., E_g\}$.  Put another way, there is a collar between $E_*$ and ${E_*}'$ in $B^3$, a collar that contains both the trivial arc $I_0$ and the parallelism disk $E_0$ but is disjoint from $\{E_2, E_3, ..., E_g\}$.  See Figure \ref{fig:drop1}.  As before, let $P$ denote the planar surface $\bdd H - Y$, that is the planar surface obtained from $\bdd H$ by deleting a neighborhood of the meridian disks $\{E_1, E_2, ..., E_g\}$.  Modeling the concrete description of generators of $\anch \subset B_2(P)$ given via Figure \ref{fig:AgenP}, the definition of $\anch_{\{E_1, ..., E_g\}}$ begins with a collection of loops $\gamma_i, \gamma_i' \subset P, 1 \leq i \leq g$ so that all the loops are mutually disjoint, except in their common end points at $x_1$; for each $i$, one of $\gamma_i$ and $\gamma_i'$ is parallel in $P$ to each of the two copies of $\bdd E_i$ in $\bdd P$; each loop is disjoint from $\bdd E_0$ except at $x_1$; and (what is new) each loop intersects only one end of the collar that lies between $\bdd E_*$ and $\bdd {E_*}'$.  Then $2g$ generators $\mathfrak{a}_i, \mathfrak{a}_i', 1 \leq i \leq g$ of $\anch_{\{E_1, ..., E_g\}}$ are represented by isotopies obtained by sliding the endpoint $x_1$ of $I_0$ around the loops $\gamma_i$ and $\gamma_i'$ respectively. 

If $\gamma_i$ (resp $\gamma_i'$) is one of the loops disjoint from $E_*$, then $\mathfrak{a}_i$ (resp $\mathfrak{a}_i'$) also lies in $\anch_{\{E_*, ..., E_g\}}$.  If, on the other hand, $\gamma_i$ (resp $\gamma_i'$) is one of the loops that is disjoint from ${E_*}'$, then $f \mathfrak{a}_i f^{-1}$ (resp $f \mathfrak{a}_i' f^{-1}$) is represented  by an isotopy of $I_0$ that is disjoint from $E_*$.  Moreover, since an isotopy of $I_0$ representing $f$ doesn't disturb the disks $\{E_2, E_3, ..., E_g\}$, the isotopy of $I_0$ representing $f \mathfrak{a}_i f^{-1}$ (resp $f \mathfrak{a}_i' f^{-1}$) is disjoint from these disks as well.  That is, each such $f \mathfrak{a}_i f^{-1}$ (resp $f \mathfrak{a}_i' f^{-1}$)  lies in $\anch_{\{E_*, ..., E_g\}}$.  Hence in all cases, $\mathfrak{a}_i$ (resp $\mathfrak{a}_i'$) lies in $\AF_{\{E_*, E_2, ..., E_g\}}$, so $\anch_{\{E_1, ..., E_g\}} \subset \AF_{\{E_*, E_2, ..., E_g\}}$.
\end{proof}

It is well-known that in a genus $g$ handlebody any two complete collections of $g$ meridian disks can be connected by a sequence of complete collections of meridian disks so that at each step in the sequence a single meridian disk is replaced with a different and disjoint one. See, for example, \cite[Theorem 1]{Wa}.   It follows then from Lemma \ref{lemma:drop1} that the subgroup $\AF_{\{E_1, E_2, ..., E_g\}}$ is independent of the specific collection of meridian disks, so we can simply denote it $\AF$.

\begin{thm} The inclusion $\AF \subset \pi_1(\Unk(I, H))$ is an equality. \label{thm:main2}
\end{thm}
\begin{proof}  Begin with a parallelism disk $E_0$ for $I_0$, with $\aaa_0$ the arc $E_0 \cap \bdd H$ connecting the endpoints $x_0$ and $x_1$.  Suppose $\ooo \in \pi_1(\Unk(I, H))$ is represented by a proper isotopy $f_t: I \to H$ extending to the ambient isotopy $f_t: H \to H$.  Adjust the end of the ambient isotopy so that $E_0$ and $f_1(E_0)$ are in normal position. Let $\aaa_{\ooo} \subset \bdd H$ denote the image $f_1(\aaa_0)$, an arc in $\bdd H$ that also connects $x_0$ and $x_1$.  The proof is by induction on $|\aaa_0 \cap \aaa_{\ooo}|$, the number of points in which the interiors of $(\aaa_0)$ and $ \aaa_{\ooo}$ intersect.  The number is always even, namely twice the number $|E_0 \cap f_1(E_0)|$ of arcs of intersection of the disk interiors.  

If $|\aaa_0 \cap \aaa_{\ooo}| = 0$, so $\aaa_0$ and $\aaa_{\ooo}$ intersect only in their endpoints at $x_0$ and $x_1$, then the union of $\aaa_0$ and $\aaa_{\ooo}$ is a simple closed curve in $\bdd H$ that bounds a (possibly inessential) disk $E_{\cup}$.  The disk $E_{\cup}$ properly contains $I_0$ and is properly contained in $H$.  Choose a complete collection of meridian disks $\{E_1, ..., E_g\}$ for $H$ that is disjoint from $E_{\cup}$.  Then an isotopy of one end of $I_0$ completely around $\bdd E_{\cup}$ (pushed slightly aside) represents an element $\mathfrak{b}$ of $\anch_{\{E_1, ..., E_g\}}$ and takes $\aaa_{\ooo}$ to  $\aaa_0$. See the earlier Figure \ref{fig:Ecup}.   It follows from Lemma \ref{lemma:diskback} that $\ooo \mathfrak{b} \in \free_{E_0}$ so $\ooo \in \AF$.

For the inductive step, suppose that any element in $\pi_1(\Unk(I, H))$ whose corresponding isotopy carries some $I_0$-parallel arc $\aaa$ to an arc that intersects $\aaa$ in $k$ or fewer points is known to lie in $\AF$. Suppose also that $|\aaa_0 \cap \aaa_{\ooo}| = k + 2$.    Let $E^*$ be a disk in $f_1(E_0) - E_0$ cut off by an outermost arc $\bbb$ of $E_0 \cap f_1(E_0)$ in $f_1(E_0)$.  Then attaching $E^*$ along $\bbb$ to the component of $E_0 - \bbb$ that contains $I_0$ gives a parallelism disk $F_0$ that is disjoint from $E_0$ and intersects $\aaa_{\ooo}$ in $\leq k$ points.  The union of $F_0$ and $E_0$ along $I_0$ is a properly embedded disk $E_{\cup}$ and, as usual, there is an isotopy representing an element $\eta$ of $\AF$ that carries $F_0$ to $E_0$.  Now apply the inductive assumption to the product $\eta \ooo$:  The isotopy corresponding to $\eta \ooo$ carries the arc $F_0 \cap \bdd H$ to $\aaa_{\ooo}.$   It follows by inductive assumption then that $\eta \ooo \in \AF$.  Hence also $\ooo \in \AF$, completing the inductive step. 
\end{proof}

Theorem \ref{thm:main1} is then an obvious corollary, and provides an explicit set of generators for $\pi_1(\Unk(I, H)).$

\section{Connection to width}

Suppose $f_t: I \to H$ is a proper isotopy from $I_0$ back to itself, representing an element $\ooo \in \pi_1(\Unk(I, H))$.  Put $f_t$ in general position with respect to the collection $\Delta$ of meridian disks $\{ E_1, ... E_g \}$ so $f_t$ is transverse to $\Delta$ (in particular, $f(\bdd I) \cap \Delta = \emptyset$) at all but a finite number $0 <  t_0 \leq t_1 \leq ... \leq t_n < 1$ of values of $t$.  Let $c_i, 1 \leq i \leq n$ be any value so that $t_{i - 1} \leq c_i \leq t_i$ and define $w_i = |f_{c_i}^{-1}(\Delta)| =  |f_{c_i}(I) \cap \Delta|$ to be the width of $I$ at $c_i$.  The values $w_i, 1 \leq i \leq n$ are all independent of the choice of the points $c_i \in (t_{i - 1},  t_i)$ since the value of $|f_{t}(I) \cap \Delta|$ can only change at times when $f_t$ is not transverse to $\Delta$.  

\begin{defin}  The width $w(f_t)$ of the isotopy $f_t$ is $max_i \{ w_i \}$.  The width $w(\ooo)$ of $\ooo \in \pi_1(\Unk(I, H))$ is the minimum value of $w(f_t)$ for all isotopies $f_t$ that represent $\ooo$.
\end{defin}

\begin{cor} \label{cor:product} For any $\ooo_1, \ooo_2 \in \pi_1(\Unk(I,H))$, the width of the product $w(\ooo_1\ooo_2) \leq max\{ w(\ooo_1), w(\ooo_2) \}.$
\end{cor}

It follows that, for any $n \geq 0$, the set of elements of $\pi_1(\Unk(I,H))$ of width no greater than $n$ constitutes a subgroup of $\pi_1(\Unk(I,H))$.  For example, if $n = 0$ the subgroup is precisely the image of the inclusion-induced homomorphism $\pi_1(\Unk_P(I, B^3)) \to \pi_1(\Unk(I, H))$ defined at the beginning of Section \ref{sect:unkinH}, a subgroup that includes the anchored subgroup $\mathfrak{A}_{\Delta}$.  It is easy to see that the width of any element in $ \mathfrak{F}_{E_0}$ is at most $1$, so it follows from Corollary \ref{cor:product} and Theorem \ref{thm:main1}, that the width of any element in $\pi_1(\Unk(I, H)) =  \mathfrak{AF}_{\{E_0, E_1, ..., E_g\}}$ is at most $1$.  

The fact that every element in $\pi_1(\Unk(I, H))$ is at most width $1$ suggests an alternate path to a proof of Theorem \ref{thm:main1}, a path that would avoid the technical difficulties of Lemmas \ref{lemma:drop0} and \ref{lemma:drop1}: prove directly that any element of width $1$ is in $\mathfrak{AF}_{\{E_0, E_1, ..., E_g\}}$ and prove directly  that any element in $\pi_1(\Unk(I, H))$ has width at most $1$ (say by thinning a given isotopy as much as possible) .  We do so below.  Neither argument requires a change in the meridian disks ${\{E_1, ..., E_g\}}$. The first argument, that any element of width $1$ is in $\mathfrak{AF}_{\{E_0, E_1, ..., E_g\}}$, is the sort of argument that might be extended to isotopies of unknotted graphs in $H$, not just isotopies of the single unknotted arc $I$, just as other thin position arguments have been extended to graphs (see \cite{ST1}, \cite{ST2}).  The second argument, which shows that the thinnest representation of any element in $\pi_1(\Unk(I, H))$ is at most width $1$, seems difficult to generalize to isotopies of an arbitrary unknotted graph, because the argument doesn't directly thin a given isotopy, but rather makes use of Lemma \ref{lemma:diskback} in a way that may be limited to isotopies of a single arc.  

\begin{prop}  If an element $\ooo \in \pi_1(\Unk(I, H))$ has $w(\ooo) = 1$ then $\ooo \in \mathfrak{AF}_{\{E_0, E_1, ..., E_g\}}$
\end{prop}

\begin{proof} Suppose $w(\ooo) = 1$ and $f_t: I \to H$ is an isotopy realizing this width.  Let $0 <  t_0 \leq t_1 \leq ... \leq t_n < 1$ be the critical points of the isotopy and, as defined above, let $w_i$ be the width of $I$ during the $i^{th}$ interval. Since each $w_i$ is either $0$ or $1$ and $w_{i} \neq w_{i-1}$ it follows that the value of $w_i$ alternates between $0$ and $1$.  Since $I_0 \cap \Delta = \emptyset$ the value of $w_1 = 1$.  During those intervals when the width is $0$, $f_t(I)$ is disjoint from $\Delta$, so the isotopy $f_t$ can be deformed, without altering the width, so that  at some time $t$ during each such interval, $f_t(I) = I_0$.  Thereby $\ooo$ can be viewed as the product of elements, for each of which there is an isotopy with just two critical points, $t_0, t_1$.   So we henceforth can assume that $n = 1$ and there are just two critical points during the isotopy.   If either point were critical because of a tangency between $f_{t_i}(I)$ and $\Delta$, then the value of $w_i$ would change by $2$ as the tangent point passed through $\Delta$, and this is impossible by the assumption in this case.  We conclude that $t_0$ and $t_1$ are critical because they mark the point at which $f_t$ moves a single end of $I$ through $\Delta$, say through the disk $E_1 \subset \Delta$.  It is possible that at $t_0$ and $t_1$ the same end of $I$ moves through $E_1$ (so $I$ does not pass completely through $E_1$), or possibly different ends of $I$ move through $E_1$ (when $I$ does pass completely through $E_1$).  We next show that an isotopy of the first type can be deformed to a sequence of two isotopies of the second type.  

To that end, suppose that for $t$ between $t_0$ and $t_1$, $f_t(I)$ intersects $\Delta$ in a single point, lying in the disk $E_1 \subset \Delta$, that at $t_0$ an end of $I$ first passes through $E_1$ and at $t_1$ the same end of $I$ is passed back through $E_1$, eliminating the intersection point with $\Delta$.   Pick a value $t'$ between $t_0$ and $t_1$; a standard innermost disk, outermost arc argument shows that there is a parallelism disk $D$ for $f_{t'}(I)$ that $E_1$ intersects in a single arc, so $D$ is divided into two subdisks.  One of these subdisks $D_-$ is incident to the end of $I$ that never passes through $E_1$.  Use the disk $D_-$ to deform $f_t$, first isotoping the subarc $f_{t'}(I) \cap D_-$ through $E_1$ and then isotoping it back to its original position.   Deforming $f_t$ further as in the previous paragraph, so that $f_t(I)$ briefly returns to $I_0$ without intersecting $\Delta$, the resulting isotopy still represents $\ooo$ but now can be viewed as the product of two elements, each still represented by an isotopy with just two critical points, but during which $I$ passes all the way through $E_1$.  Hence it suffices to henceforth to assume that at $t_0$ and $t_1$, different ends of $I$ pass through $E_1$ because during the isotopy $I$ passes completely through $E_1$.  

\bigskip

{\bf Special Case:} There is a parallelism disk $D$ for $I$ so that the interiors of both $f_{t_0}(D)$ and $f_{t_1}(D)$ are disjoint from the meridian disks $\Delta$.

In this case, replace the isotopy $f_t, 0 \leq t \leq t_0$ by an isotopy that carries $E_0$ to $f_{t_0}(D)$ in $H - \Delta$ and replace the isotopy $f_t, t_1 \leq t \leq 1$ by an isotopy that carries $f_{t_1}(D)$ to $E_0$  in $H - \Delta$.  The resulting isotopy carries $E_0$ back to itself, and therefore represents an element of $\free_{E_0}$.  It has been obtained by pre- and post-multiplying $\ooo$ by elements whose representing isotopies keep $I$ disjoint from $\Delta$ and therefore lie in $\AF_{E_0, ..., E_g}$.  It follows that $\ooo \in  \AF_{E_0, ..., E_g}$ as required.  

\bigskip

{\bf General case:}  

$f_{t_0}(I)$ intersects $E_1$ exactly in a single endpoint.  Since the isotopy is generic, at time $t_0 + \epsilon$,  there is a small disk $D_0$ on the side of $E_1$ opposite to $f_{t_0}(I)$ so that $\bdd D_0$ is the union of an arc in $E_1$, an arc in $\bdd H$ and the small  end-segment of $f_{t_0 + \epsilon}(I)$.  Similarly, at time $t_1 - \epsilon$ there is a small disk $D_1$ on the other side of $E_1$ with $\bdd D_1$ the union of an arc in $E_1$, an arc in $\bdd H$ and the small end-segment of $f_{t_1 - \epsilon}(I)$.  Pick a generic point $t_0 + \epsilon < t' < t_1 - \epsilon$.  For $t$ between $t_0 + \epsilon$ and $t_1 - \epsilon$, the arc $f_t(I)$ is always divided into two segments by $E_1$.  Apply the isotopy extension theorem to $D_0$ in the interval $[t_0 + \epsilon, t']$ and to $D_1$ in the interval $[t', t_1 - \epsilon]$ to get two disks at time $t'$, say $D_0'$ and $D_1'$, one on either side of $E_1$, each intersecting $E_1$ in an embedded arc with one end at the point $f_t(I') \cap E_1$ and the other end on $\bdd E_1$.  These two arcs in $E_1$ do not necessarily coincide, nor are the disks $D_0'$ and $D_1'$ necessarily disjoint, but $f_{t'}(I)$ is disjoint from the interiors of both of the disks.  It follows that $D_1'$ can be isotoped rel  $f_{t'}(I)$ so that its interior is disjoint from $D_0'$ and so that its boundary intersects $E_1$ in the same arc that $\bdd D_0'$ does.  This isotopy of $D_1'$ can be absorbed into the ambient extension to all of $H$ of the isotopy $t_t(I)$ near time $t'$, altering $D_1'$ so that the two arcs $\bdd D_0' \cap E_1$ and $\bdd D_1' \cap E_1$ do coincide and the interiors of $D_0'$ and $ D_1'$ are disjoint .  The union of $D_0'$ and $ D_1'$ along their common boundary arc of intersection with $E_1$ then provides a parallelism disk $D$ for which the special case above applies.   
\end{proof}

\begin{prop}  Any element $\ooo \in \pi_1(\Unk(I, H))$ has a representative isotopy which is of width at most $1$ with respect to the collection of meridian disks $\Delta$.  
\end{prop}
  
\begin{proof}  Let $D_0$ be any disk of parallelism for $I_0$ that is disjoint from $\Delta$, let $f_t$ be an isotopy of $I$ that represents $\ooo$, so in particular $f_1(I_0) = I_0$.  Extend $f_t$ to an ambient isotopy of $H$, and let $D_1 = f_1(D_0)$ be the final position of $D_0$ after the isotopy.  Call $D_1$ the {\em terminal disk} for the isotopy.  It follows from Lemma \ref{lemma:diskback} that, up to a product of elements of $\free_{D_0}$, each of which has width $1$ with respect to $\Delta$, $\ooo$ is represented by any isotopy of $I_0$ back to itself that has the same terminal disk $D_1$.  The proof will be by induction on $|D_1 \cap \Delta|$; we assume this has been minimized by isotopy, so in particular all components of intersection are arcs with end points on the arc $\aaa = \bdd D_1 \cap \bdd H$.  

Given the disk $D_1$, here is a useful isotopy, called a {\em $D_1$-sweep}, that takes $D_0$ to $D_1$. Pick any $p$ point on $\aaa$, and isotope $D_0$ in the complement of $\Delta$ so that it becomes a small regular neighborhood $N_p$ of $p$ in $D_1$. (Do not drag $D_1$ along.) This isotopy carries $I_0$ to an arc $I_p = \bdd N_p - \bdd H$ properly embedded in $D_1$.  Then stretch $N_p$ in $D_1$ until it fills all of $D_1$, so we view $N_p$ as sweeping across all of $D_1$.  The combination of the two isotopies carries $I_0$ back to itself and takes $D_0$ to $D_1$ so, up to a further product with width one isotopies, we can assume that this combined isotopy represents $\ooo$ (via Lemma \ref{lemma:diskback}).  

It will be useful to have a notation for this two stage process: let $g$ be the isotopy in the complement of $\Delta$ that takes $D_0$ to $N_p$ and let $s$ be the isotopy that sweeps $N_p$ across $D_1$.  We wish to determine the width of the sequence $g*s$ of the two isotopies.  Since the isotopy $g$ is disjoint from $\Delta$, much depends on the $D_1$-sweep $s$ of $N_p$ across $D_1$.  In particular, if the number of intersection arcs $|D_1 \cap \Delta| \leq 1$ it is obvious how to arrange the sweep so the width is at most $1$.  So henceforth we assume that $|D_1 \cap \Delta| \geq 2$.

 \begin{figure}[ht!]
 \labellist
\small\hair 2pt
\pinlabel $g$ at 50 170
\pinlabel $N_p$ at 185 25
\pinlabel $\bbb'$ at 160 50
\pinlabel $\bbb$ at 205 50
\pinlabel $s$ at 185 45
\pinlabel $D_0$ at 10 100
\pinlabel $D_1$ at 125 40
\pinlabel \mbox{$D_p \subset D_1$} at 235 90
 \endlabellist
    \centering
    \includegraphics[scale=0.9]{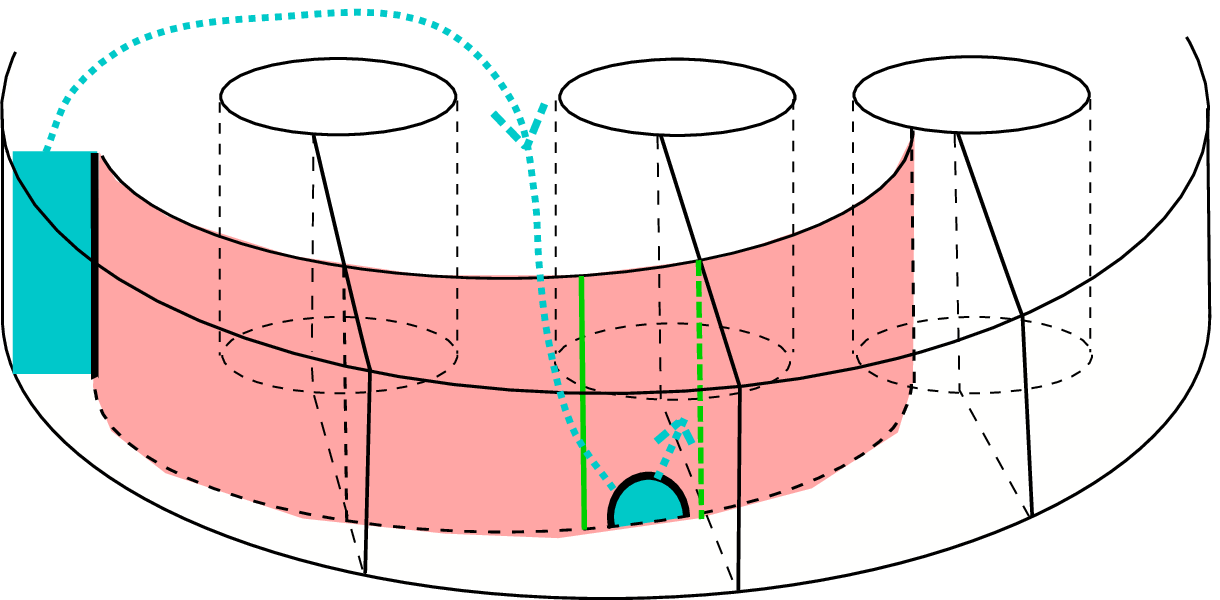}
\caption{} \label{fig:widthone}
    \end{figure}

Of all arcs in $D_1 \cap \Delta$, let $\beta$ be one that is outermost in $\Delta$.  
That is, the interior of one of the disks that $\bbb$ cuts off from $\Delta$ is disjoint from $D_1$.  
Pick the point $p$ for the $D_1$-sweep to be near an end-point of $\bbb$ in $\bdd D_1$, on a side of $\bbb$ that contains at least one other arc.  (Since $|D_1 \cap \Delta| \geq 2$ there is at least one other arc.)  Let $\bbb' \subset D_1$ be an arc in $D_1 - \Delta$ with an end at $p$ and which is parallel to $\bbb$ and let $D_p$ be the subdisk of $D_1$ cut off by $\bbb'$ that does not contain $I_0 \subset \bdd D_1$.  
See Figure \ref{fig:widthone}.  Now do the sweep $s$ in two stages:  first sweep $N_p$ across $D_p$ until it coincides with $D_p$, then complete the sweep across $D_1$.  See Figure \ref{fig:sweeptwo}
Denote the two stages of the sweep by $s = s_1 * s_2$.  The isotopy $s_1$ carries $I_p$ to the arc $\bbb'$; exploiting the fact that $\bbb$ is outermost in $\Delta$ there is an obvious isotopy $h$ (best imagined in Figure \ref{fig:widthone}) that carries $\bbb'$ back to $I_p$, an isotopy that is disjoint from $\Delta$ and from $D_1$.  Finally, deform the given isotopy $g*s = g*s_1*s_2$ whose width we seek, to the isotopy $g*s_1*h*\overline{g}*g*\overline{h}*s_2$, which can be written as the product $(g*s_1*h*\overline{g})*(g*\overline{h}*s_2)$, of two isotopies, each representing an element of $\pi_1(\Unk(I, H))$.  The first isotopy has terminal disk containing exactly the arcs of $D_p \cap \Delta$ and the second has terminal disk containing the other arcs $D_1 \cap \Delta$.  By construction of $\bbb'$ each set is non-empty so each terminal disk intersects $\Delta$ in fewer arcs than $D_1$ did.  By inductive hypothesis, each isotopy can be deformed to have width at most $1$ so, by Corollary \ref{cor:product}, $w(\ooo) \leq 1$.  \end{proof}

 \begin{figure}[ht!]
 \labellist
\small\hair 2pt
\pinlabel $I_0$ at 10 170
\pinlabel $N_p$ at 110 132
\pinlabel $\bbb'$ at 90 160
\pinlabel $\bbb$ at 125 160
\pinlabel $s_1$ at 135 40
\pinlabel $s_2$ at 75 65
\pinlabel \mbox{$D_p$} at 170 40
 \endlabellist
    \centering
    \includegraphics[scale=1.0]{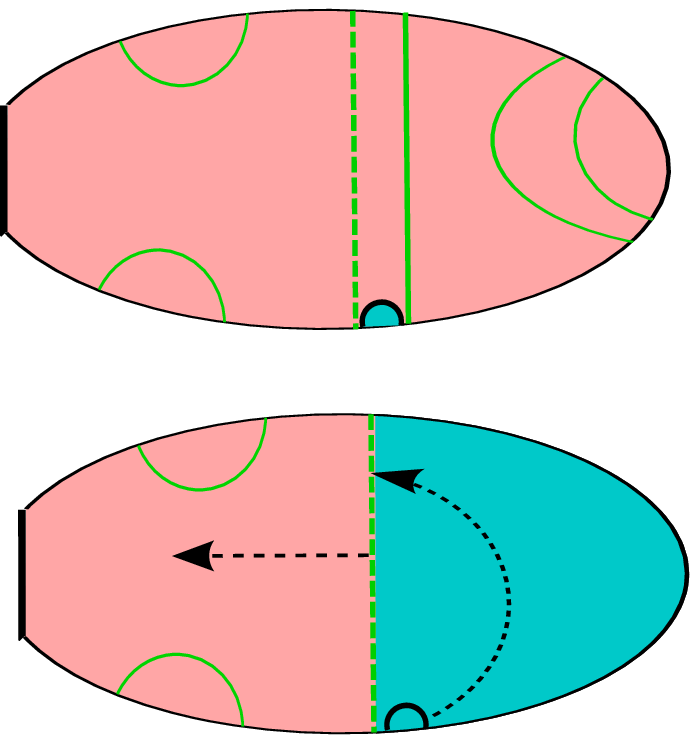}
\caption{} \label{fig:sweeptwo}
    \end{figure}

\section{Connection to the Goeritz group} \label{sect:toGoeritz}

Suppose $H$ is a genus $g \geq 1$ handlebody, and $\Sss$ is a genus $g+1$ Heegaard surface in $H$.  That is, $\Sss$ splits $H$ into a handlebody $H_1$ of genus $g + 1$ and a compression-body $H_2$.  $H_2$ is known to be isotopic to the regular neighborhood of the union of $\bdd H$ and an unknotted properly embedded arc $I_0 \subset H$ (see \cite[Lemma 2.7]{ST1}).  

Let $\Diff(H)$ be the group of diffeomorphisms of $H$ and $\Diff(H, \Sss) \subset \Diff(H)$ be the subgroup of diffeomorphisms that take the splitting surface $\Sss$ to itself.  Following \cite{JM},  define the Goeritz group $G(H, \Sss)$ of the Heegaard splitting to be the group consisting of those path components of $\Diff(H, \Sss)$ which lie in the trivial path component of $\Diff(H)$.  So, a non-trivial element of $G(H, \Sss)$ is represented by a diffeomorphism of the pair $(H, \Sss)$ that is isotopic to the identity as a diffeomorphism of $H$, but no isotopy to the identity preserves $\Sss$.  

\begin{thm} \label{thm:goeritz}  For $H$ a handlebody of genus $\geq 2$, 
$$G(H, \Sss) \cong\pi_1(\Unk(I, H)).$$ 

For $H$ a solid torus, there is an exact sequence 
$$1 \to \zed \to \pi_1(\Unk(I, H)) \to G(H, \Sss) \to 1.$$

In either case, the finite collection of generators of $\pi_1(\Unk(I, H))$ described above is a complete set of generators for $G(H, \Sss)$.
\end{thm}

\begin{proof} This is a special case of \cite[Theorem 1]{JM}, but there the fundamental group of the space $\H(H, \Sigma) = \Diff(H)/\Diff(H, \Sss)$ takes the place of $\pi_1(\Unk(I, H))$.  So it suffices to show that there is a homotopy equivalence between $\H(H, \Sigma)$ and $\Unk(I, H)$.  We sketch a proof:

\medskip

Fix a diffeomorphism $e: I \to I_0 \subset H$, for $I_0$ as above a specific unknotted arc in $H$.  It is easy to see that for any other embedding $e': I \to H$ with $e'(I)$ unknotted, there is a diffeomorphism $h: H \to H$, so that $he = e'$.  It follows that restriction to $I_0$ defines a surjection $\Diff(H) \to \Emb_0(I, H)$.  Since any automorphism of $I_0$ extends to an automorphism of $H$, this surjection maps the subgroup $\Diff(H, I_0)$   (diffeomorphisms of $H$ that take $I_0$ to itself) onto the space of automorphisms of $I_0$. It follows that $\Unk(I, H) = \Emb_0(I, H)/\Diff(I)$ has the same homotopy type as $\Diff(H)/\Diff(H, I_0).$ 

Suppose for the compression-body $H_2$ in the Heegaard splitting above we take a fixed regular neighborhood of $\bdd H \cup I_0 \subset H$.  Let $\Diff(H, H_2, I_0)$ denote the subgroup of $\Diff(H, H_2)$ that takes $I_0$ to itself.  

\begin{lemma} \label{lemma:I0inH2} The inclusion $\Diff(H, H_2, I_0) \subset \Diff(H, H_2)$ is a homotopy equivalence.  
\end{lemma}

\begin{proof}  First note that the path component $\Diff_0(H, H_2)$ of $\Diff(H, H_2)$ containing the identity is contractible, using first \cite{EE} on $\bdd H_2$ and then \cite{Ha1} on the interiors of $H_1$ and $H_2$.  Any other path component of $\Diff(H, H_2)$ is therefore contractible, since each is homeomorphic to $\Diff_0(H, H_2)$. It therefore suffices to show that any path component of $\Diff(H, H_2)$ contains an element of $\Diff(H, H_2, I_0)$.  Equivalently, it suffices to show that any diffeomorphism of $H_2$ is isotopic to one that sends $I_0$ to $I_0$.  A standard innermost disk, outermost arc argument shows that $H_2$ contains, up to isotopy, a single $\bdd$-reducing disk $D_0$; once $D_0$ has been isotoped to itself and then the point $I_0 \cap D_0$ isotoped to itself in $D_0$, the rest of $I_0$ can be isotoped to itself, using a product structure on $H_2 - \eta(D_0) \cong \bdd H \times I$.
\end{proof}

\begin{lemma} \label{lemma:I0inH}The inclusion $\Diff(H, H_2, I_0) \subset \Diff(H, I_0)$ is a homotopy equivalence.  
\end{lemma}

\begin{proof}  Such spaces have the homotopy type of $CW$-complexes.  (See \cite[Section 2]{HKMR} for a discussion of this and associated properties.)  So it suffices to show that, for any $$\Theta: (B^k, \bdd B^k) \to (\Diff(H, I_0), \Diff(H, H_2, I_0)), k \geq 1,$$ there is a pairwise homotopy  to a map whose image lies entirely inside $\Diff(H, H_2, I_0)$.  The core of the proof is a classic argument in smooth topology; a sketch of the rather complex argument is given in the Appendix. Note that the level of analysis that is required is not much deeper than the Chain Rule in multivariable calculus.\footnote{But in some ways the argument in the Appendix is just a distraction:  First of all, we do not need the full homotopy equivalence to show that the fundamental groups of $\H(H, \Sigma)$ and $\Unk(I, H)$ are isomorphic (but restricting attention just to the fundamental group wouldn't really simplify the argument).  Secondly, we could have, from the outset, informally viewed each isotopy of $I$ in $H$ described in the proof above to be a proxy for an isotopy of a thin regular neighborhood $H'$ of $\bdd H \cup I$;  the difference then between isotopies of $H'$ and isotopies of the compression body $H_2$ is easily bridged, requiring from the Appendix only Lemma \ref{lemma:Hatch} and following.  }
\end{proof} 

A diffeomorphism of $H$ that takes $\Sigma$ to itself will also take $H_2$ to itself, since  $H_2$ and $H_1$ are not diffeomorphic.  Hence $\Diff(H, \Sss) = \Diff(H, H_2)$ and $$\H(H, \Sigma) = \Diff(H)/\Diff(H, \Sss) =  \Diff(H)/\Diff(H, H_2).$$

Lemma \ref{lemma:I0inH2} shows that the natural map $$\Diff(H)/\Diff(H, H_2, I_0) \to \Diff(H)/\Diff(H, H_2)$$ is a homotopy equivalence and Lemma \ref{lemma:I0inH} shows that the natural map $$\Diff(H)/\Diff(H, H_2, I_0) \to \Diff(H)/\Diff(H, I_0)$$ is a homotopy equivalence.  Together these imply that $\H(H, \Sigma) =  \Diff(H)/\Diff(H, H_2)$ is homotopy equivalent to $\Diff(H)/\Diff(H, I_0)$, which we have already seen is homotopy equivalent to $\Unk(I, H)$.  
 \end{proof}
 
 \newpage

\appendix
\section{Moving a diffeomorphism to preserve $H_2$.}

A conceptual sketch of the proof of Lemma \ref{lemma:I0inH} can be broken into six steps.  It is useful to view $\Theta: (B^k, \bdd B^k) \to (\Diff(H, I_0), \Diff(H, H_2, I_0))$ as a family of diffeomorphisms $h_u:(H, I_0) \to (H, I_0)$ parameterized by $u \in B^k$ so that for $u$ near $\bdd B^k$, $h_u(H_2) = H_2.$  By picking a particular value $u_0 \in \bdd B^k$ and post-composing all diffeomorphisms with $h_{u_0}^{-1}$ we may, with no loss of generality, assume that $\Theta(B^k)$ lies in the component of $\Diff(H, I_0)$ that contains the identity.  Here are the six stages:

\begin{enumerate}
\item  (Pairwise) homotope $\Theta$ so that for each $u$, 
$h_u$ is the identity on $\bdd H$.
\item Further homotope $\Theta$ so that for some collar structure $\bdd H \times I$ near $\bdd H$ and for each $u \in B^k$, $h_u|(\bdd H \times I)$ is the identity.
\item Further homotope $\Theta$ so that after the homotopy there is a neighborhood of $I_0$ and a product structure $I_0 \times \mbbR^2$ on that neighborhood so that for all $u \in B^k$, $h_u|(I_0 \times \mbbR^2)$ commutes with projection to $I_0$ near $I_0$.
\item Further homotope $\Theta$ so that for all $u \in B^k$, $h_u|(I_0 \times \mbbR^2)$ is a linear (i. e. $\GL_n$) bundle map over $I_0$ near $I_0$.
\item Further homotope $\Theta$ so that for all $u \in B^k$, $h_u|(I_0 \times \mbbR^2)$ is an orthogonal (i. e. $\Ooo_n$) bundle map near $I_0$. There is then a regular neighborhood $H' \subset H_2$ of $\bdd H \cup I_0$ so that for all $u \in B^k$, $h_u(H') = H'$.
\item Further homotope $\Theta$ so that for the specific regular neighborhood $H_2$ of $\bdd H \cup I_0$ and for each $u \in B^k$, $h_u(H_2) = H_2$.
\end{enumerate}

For the first stage, let $\Theta_{\bdd H}: B^k \to \Diff(\bdd H, \bdd I_0)$ be the restriction of each $h_u$ to $\bdd H$.  The further restriction $\Theta_{\bdd H}|\bdd B^k$ defines an element of $\pi_{n-1}(\Diff(\bdd H, \bdd I_0))$ in the component containing the identity $h_{u_0}|\bdd H$. The space $ \Diff(\bdd H, \bdd I_0)$ is known to be contractible (see \cite{EE}, \cite{ES}, or \cite{Gr}).  Hence $\Theta_{\bdd H}|\bdd B^k$ is null-homotopic.  The homotopy can be extended to give a homotopy of $\Theta_{\bdd H}:B^k \to \Diff(\bdd H, \bdd I_0)$ and then, by the isotopy extension theorem, to a homotopy of $\Theta: B^k \to \Diff(H, I_0)$, after which each $h_u|\bdd H$ is the identity.

The homotopy needed for the second stage is analogous to (and much simpler than) the sequence of homotopies constructed in stages $3$ to $5$, so we leave its construction to the reader.  

Stages 3 through 5 might best be viewed in this context:  One is given a smooth embedding $h: I \times \mbbR^2 \to I \times \mbbR^2$ that restricts to a diffeomorphism $I \times \{0 \} \to I \times \{0 \}$ and which is the identity near $\bdd I \times \mbbR^2$.  One hopes to isotope the embedding, moving only points near $I \times \{0 \}$ and away from $\bdd I \times \mbbR^2$ so that afterwards the embedding is an orthogonal bundle map near $I \times \{0 \}$.  Moreover, one wants to do this in a sufficiently natural way that a $B^k$-parameterized family of embeddings gives rise to a  $B^k$-parameterized family of isotopies.  The relevant lemmas below are proven for general $\mbbR^n$, not just $\mbbR^2$, since there is little lost in doing so. (In fact the arguments could easily be further extended to families of smooth embeddings $D^{\ell} \times \mbbR^n \to D^{\ell} \times \mbbR^n, \ell > 1$).

\bigskip

\noindent {\bf Stage 3: Straightening the diffeomorphism near $I_0$} \label{subs:hstraight}

\begin{defin} A smooth isotopy $f_t: I \times \mbbR^n \to I \times \mbbR^n$ of an embedding $f_0: I \times \mbbR^n \to I \times \mbbR^n$ is {\rm allowable} if it has compact support in $(int\; I) \times \mbbR^n$.  That is, the isotopy is fixed near $\bdd I \times \mbbR^n$ and outside of a compact set in $I \times \mbbR^n$.  
\end{defin}

\begin{defin} A smooth embedding $f: I \times (\mbbR^n, 0) \to I \times  (\mbbR^n, 0)$ commutes with projection to $I$ if for all $(x, y) \in I \times \mbbR^n$, $p_1f(x, y) = f(x, 0) \in I$.  
\end{defin}

\begin{defin} For $A$ a square matrix or its underlying linear transformation, let $|A|$ denote the operator norm of $A$, that is \begin{align*}|A| &= max\{|Ax|: x \in \mbbR^n, |x| = 1\} \\ &= max\{\frac{|Ax|}{|x|}: 0 \neq x \in \mbbR^n\} \end{align*}.
\end{defin}

Let $\phi: [0, \infty) \to [0, 1]$ be a smooth map such that  $\phi([0, 1/2]) = 0$,  $\phi: (1/2, 1) \to (0, 1)$ is a diffeomorphism, and $\phi([1, \infty)) = 1$.  Let $b_0$ be an upper bound for $s \phi'(s)$; for example $\phi$ could easily be chosen so that $b_0 = 4$ suffices.   For any $\epsilon > 0$ let $\phi_{\epsilon}: [0, \infty) \to [0, 1]$ be defined by $\phi_{\epsilon}(s) = \phi(\frac{s}{\epsilon})$.

Then  $$\phi_{\epsilon}([\epsilon, \infty)) = 1$$ and, for any $\epsilon > 0$ and any $s \in [0, \infty)$,  $$s \phi'_{\epsilon}(s) = s \phi'(\frac{s}{\epsilon}) \frac{1}{\epsilon} = \frac{s}{\epsilon} \phi'(\frac{s}{\epsilon}) \leq b_0.$$  Thus $b_0$ is an upper bound for all  $s \phi'_{\epsilon}(s)$.  

Now fix an $\epsilon > 0$ and define $g_1: \mbbR^n \to [0, 1]$ by $g_1(y) = \phi_{\epsilon}(|y|)$.  This is a smooth function on $\mbbR^n$ which is $0$ on the ball $B_{\epsilon/2}$ and $1$ outside of $B_\epsilon$.  Linearly interpolating,  $$g_t(y) = (1-t) + tg_1(y)$$ is a smooth homotopy with support in $B_\epsilon$ from the constant function $1$ to $g_1$. 
Define a smooth homotopy $\lll_t: \mbbR^n \to \mbbR^n, 0 \leq t \leq 1$ by $\lll_t(y) = g_t(y) y$ and note that, by the chain rule, the derivative $D\lll_t(y)(z) = g_t(y)z + t\phi_{\epsilon}'(|y|)(\frac{y}{|y|} \cdot z)y$ so $D\lll_t(y)$ has the matrix $$g_t(y)I_n + t\frac{\phi'_{\epsilon}(|y|)}{|y|}yy^*$$ and so satisfies $$|D\lll_t(y)| \leq g_t(y) + t\phi'_{\epsilon}(|y|)|y|\frac{|yy^*|}{|y^2|} \leq 1 + b_0$$ since the norm of the matrix $yy^*$ is $|y|^2$.  (This last point is best seen by taking $y$ to be a unit vector, $z$ to be any other unit vector and observing that $|yy^*z| = |y(y \cdot z)| \leq |y(y \cdot y)| = 1.$  Note also that the function represented by an expression like $\phi'_{\epsilon}(|y|)/|y|$  is understood to be $0$ when $y = 0$.  Since $\phi'_{\epsilon}(y) \equiv 0$ for $y$ near $0$, the function is smooth.)

The central point of the above calculation is only this: $|D\lll_t(y)|$ has a uniform bound that is independent of $t$ or the value of $\epsilon$ that is used in the construction of $\lll$.  

\begin{lemma}[Handle-straightening] \label{lemma:hstraight} Suppose $f: I \times (\mbbR^n, 0) \to I \times 
(\mbbR^n, 0)$ is an embedding which commutes with projection to $I$  near $\bdd I \times \{0\}$.  Then there is an allowable isotopy of $f$ to an embedding that commutes with projection to $I$ near all of $I \times \{0\}$.  (See Figure \ref{fig:hstraight}.)

Moreover, given a continuous family $f^u, u \in B^k$ of such embeddings so that for $u$ near $\bdd B^k$, $f^u$ commutes with projection to $I$ near  $I \times \{0\}$, a continuous family of such isotopies can be found, and the isotopy is constant for any $u$ sufficiently near $\bdd B^k$.  
\end{lemma}

\begin{proof}
We will construct an isotopy $f_t$ for a given $f = f_0$ and then observe that it has the properties described.  By post-composing with a $GL_n$ bundle map over the diffeomorphism $f^{-1}|I \times \{0\}$, we may as well assume that $f|I \times \{0\}$ is the identity and,  along $I \times \{ 0 \}$, $D(p_2f)$ is the identity on each $\mbbR^n$ fiber $\mbbR^n_x$.  Hence, at any point $(x, 0) \in I \times \{0\}$, the matrix of the derivative $Df: \mbbR \times \mbbR^n \to  \mbbR \times \mbbR^n$ is the identity except for perhaps the last $n$ entries of the first row, which contain the gradient $\nabla (p_1f|R^n_x)$.

For any $\epsilon > 0$, consider the map $$f_t: I \times \mbbR^n \to I \times \mbbR^n$$ defined near $I \times B_\epsilon$ by 
$$f_t \left(\begin{array}{c}x \\  y\end{array}\right) = \left(\begin{array}{c}p_1f(x, \lll_t(y))\\ p_2f(x, y)\end{array}\right)$$ 
and fixed at $f$ outside $I \times B_\epsilon$.  For any value of $t$, the derivative of $f_t$ at any $(x, y) \in I \times \mbbR^n$ differs from that of $f$ by multiplying its first row on the right by the matrix 
$$\left(\begin{array}{cc}1 & 0 \\ 0 & D\lll_t(y)\end{array}\right).$$ For $\epsilon$ small, the first row of $Df_t$ will then be quite close to a row vector of the form $$\left(\begin{array}{cc}1 & v^*\end{array}\right)$$ where $|v^*| \leq (1 + b_0)|\nabla (p_1f|\mbbR^n_x)|$ and the matrix for $Df_t$ will be quite close to the matrix $$\left(\begin{array}{cc}1 & v^* \\ 0 & I_n\end{array}\right).$$ So, although we cannot necessarily make $v^*$ small by taking $\epsilon$ small, the entries in $v^*$ are at least bounded by a bound that is independent of $\epsilon$, and, by taking $\epsilon$ small, the rest of the matrix can be made to have entries arbitrarily close to those of the identity matrix.  In particular, for $\epsilon$ sufficiently small, $Df_t$ will be non-singular everywhere, and so each $f_t$ will be a smooth embedding.  Thus $f_t$ will be an allowable isotopy to a smooth embedding $f_1$ that commutes with projection to $I$ on $I \times B_\frac{\epsilon}2$, since on $I \times B_\frac{\epsilon}2$ we have  $p_1f(x, \lll_1(y)) = p_1f(x, 0) = f(x, 0) \in I.$

 \begin{figure}[ht!]
 \labellist
\small\hair 2pt
\pinlabel $I$ at 75 50
\pinlabel $\mbbR^n$ at 165 132
\pinlabel $f$ at 190 120
\pinlabel $f_1$ at 190 85
\pinlabel $B_\epsilon$ at -7 102
 \endlabellist
    \centering
    \includegraphics[scale=.9]{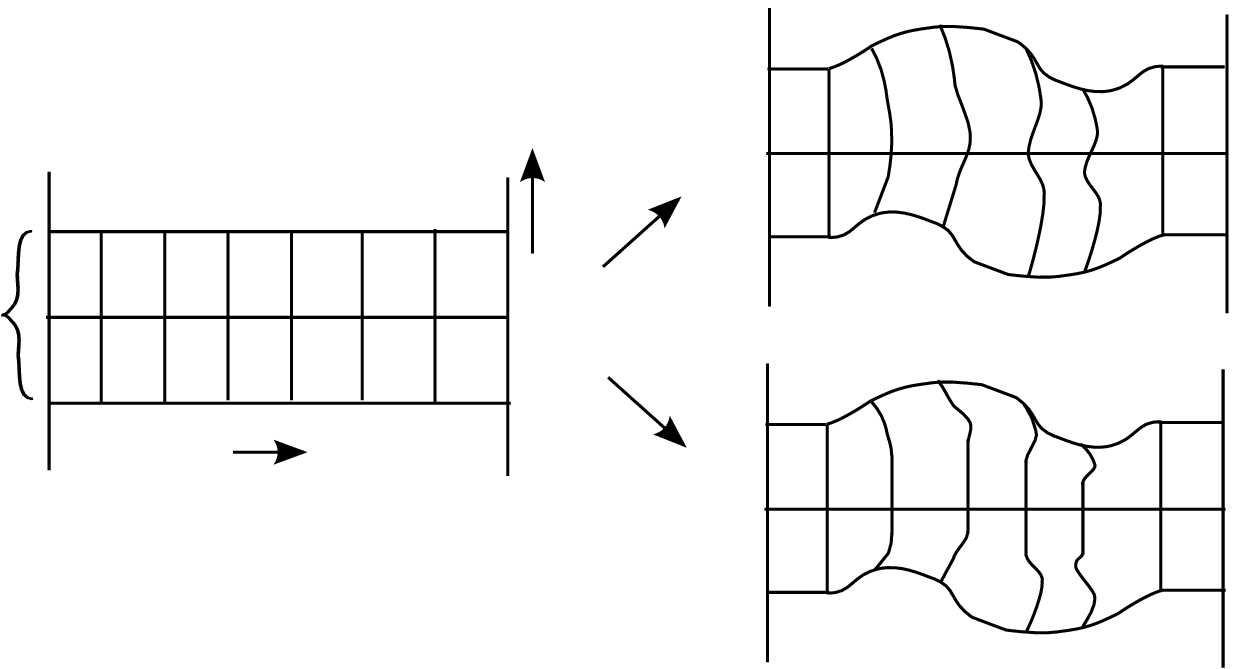}
\caption{} \label{fig:hstraight}
    \end{figure}

The extension to a parameterized family of embeddings $f^u, u \in B^k$ is relatively easy:  pick $\epsilon$ so small (as is possible, since $B^k$ is compact) so that the above argument works simultaneously on each $f^u, u \in B^k$ and also so small that, for each $u$ near $\bdd B^k$, $I \times B_{\epsilon}$ lies within the area on which $f^u$ already commutes with projection to $I$.   
\end{proof}

\noindent {\bf Stage 4: Linearizing the diffeomorphism near $I_0$}


\begin{lemma}[$\bf{Diff_n/GL_n}$] \label{lemma:DiffGL} Suppose $f: I \times (\mbbR^n, 0) \to I \times 
(\mbbR^n, 0)$ is an embedding which commutes with projection to $I$ and which is a $\GL_n$ bundle map near $\bdd I \times \{0\}$.  Then there is an allowable isotopy of $f$, through embeddings which commute with projection to $I$, to an embedding that is a $\GL_n$ bundle map near $I \times \{0\}$. 

Moreover, given a continuous family $f^u, u \in B^k$ of such embeddings so that for $u$ near $\bdd B^k$, $f^u$ is a $\GL_n$ bundle map near $I \times \{0\}$, a continuous family of such isotopies can be found, and the isotopy of $f^u$ is constant for any $u$ sufficiently near $\bdd B^k$.  
\end{lemma}

\begin{proof}
For each $x \in I$ consider the restriction $f|\mbbR^n_x$ of $f$ to the fiber $\mbbR^n_x$ over $x$.  By post-composing $f$ with the $\GL_n$ bundle map over $f^{-1}|I \times \{0\}$ determined by $D(f|\mbbR^n_x)(0)^{-1}$ we may as well assume that $f|I \times \{0\}$ is the identity and that $D(f|\mbbR^n_x)(0)$ is the identity for each $x$.  There are differentiable maps $\psi_{i,x}: \mbbR^n_x \to \mbbR^n_x, 1 \leq i \leq n$, smoothly dependent on $x$, so that for $y \in R^n_x$, $f(y) = \Sss_{i=1}^n y_i \psi_{i, x}(y)$ and $\psi_{i,x}(0) = e_i$, the $i^{th}$ unit vector in $\mbbR^n$ \cite[Lemma 2.3]{BJ}.  

Choose $\epsilon$ small and let $\lll_t: \mbbR^n \to \mbbR^n$ be the homotopy defined above using $\epsilon$.  For each $t \in [0, 1]$ define $f_{t}:I \times (\mbbR^n, 0) \to I \times (\mbbR^n, 0)$ as the bundle map for which each $f_{t}|\mbbR^n_x$ is given by $\Sss_{i=1}^n y_i \psi_{i, x}(\lll_t(y))$. When $t = 1$ the function on $B_{\epsilon/2} \subset \mbbR^n_x$ is given by $\Sss_{i=1}^n y_i e_i$, i. e. the identity.   As in Step 1, the bound $|D\lll_t(y_0)| \leq 1 + b_0$ guarantees that if $\epsilon$ is chosen sufficiently small the derivative of $\Sss_{i=1}^n y_i \psi_{i, x}(\lll_t(y))$ is close to the identity throughout $I \times B_{\epsilon}$; hence $f_t$ remains a diffeomorphism for each $t$.  

The extension to a parameterized family of embeddings $f^u$ is done as in Stage 3 (Sub-section \ref{subs:hstraight}) above.  
\end{proof}

\noindent {\bf Stage 5: Orthogonalizing the diffeomorphism near $I_0$}

\begin{lemma}[$\bf{GL_n/O_n}$] \label{lemma:GLO} Suppose $f: I \times (\mbbR^n, 0) \to I \times 
(\mbbR^n, 0)$ is a $\GL_n$ bundle map which is an $\Ooo_n$ bundle map near $\bdd I \times \{0\}$.  Then there is an allowable isotopy of $f$, through embeddings which commute with projection to $I$, to an embedding that is an $\Ooo_n$ bundle map near $I \times \{0\}$. 

Moreover, given a continuous family $f^u, u \in B^k$ of such embeddings so that for $u$ near $\bdd B^k$, $f^u$ is an $\Ooo_n$ bundle map near $I \times \{0\}$, a continuous family of such isotopies can be found, and the isotopy of $f^u$ is constant for any $u$ sufficiently near $\bdd B^k$.  
\end{lemma}


\begin{proof}
Once again we may as well assume $f|I \times \{0\}$ is the identity and focus on the linear maps $f|R^n_x$.

As a consequence of the Gram-Schmidt orthogonalization process,
any matrix $A \in GL_n$ can be written uniquely as the product $QT$ of an orthogonal matrix $Q$ and an upper triangular matrix $T$ with only positive entries in the diagonal.  The entries of $Q$ and $T = Q^tA$ depend smoothly on those of $A$.  In particular, by post-composing $f$ with the orthogonal bundle map determined by the inverse of the orthogonal part of $(Df_x)(0)$ we may as well assume that for each $x \in I$, $f|\mbbR^n_x$ is defined by an upper triangular matrix $\T_x$ with all positive diagonal entries.  

\bigskip

\noindent {\bf A worrisome example:}  It seems natural to use the function $g_t$ defined above to interpolate linearly between $\T_x$ and the identity, in analogy to the way $g_t$ (via $\lll_t$) was used in Stages 3 and 4.  Here this would mean setting  $$f_t(y)= [I_n + g_t(y)(\T_x - I_n)]y \hspace{1cm} \forall y \in R^n_x.$$ 

This strategy fails without control on the bound $b_0$ of $s \phi' (s)$, even in the relevant case $n = 2$, because $f_t$ may fail to be a diffeomorphism.   For example, if $$\T_x = \left(\begin{array}{cc}1 & r \\ 0 & 1\end{array}\right)$$ then with the above definition $$(D(f_1|\mbbR^2_x)_y)(z) = \left(\begin{array}{cc}1 & g_1(y)r \\ 0 & 1\end{array}\right)z + \left(\begin{array}{c}\phi'(|y|)(ry_2)(y\cdot z)/|y| \\ 0\end{array}\right)$$
Choose $z = \left(\begin{array}{c}1 \\ 0\end{array}\right)$ and get the vector $$\left(\begin{array}{c}1 + \phi'(|y|)(ry_1y_2)/|y| \\ 0\end{array}\right) = \left(\begin{array}{c}1 + r\phi'(|y|)|y|\frac{y_1y_2}{y_1^2 + y_2^2} \\ 0\end{array}\right)$$  For a fixed value of $|y|$, the ratio $\frac{y_1y_2}{y_1^2 + y_2^2}$ takes on every value in $[-1/2, 1/2]$.  So, unless $\phi'(|y|)|y| < \frac2{|r|}$, which could be much smaller than $b_0$, the vector $(D(f_1|\mbbR^2_x)_y)(z)$ will be trivial for some $y$. At this value of $y$, $D(f_1|\mbbR^2_x)_y$ would be singular, so $f_1|\mbbR^2_x$ would not be a diffeomorphism. 

\bigskip

On the other hand, in contrast to the previous two stages, there is no advantage to restricting the support of the isotopy $f_t$ to an $\epsilon$ neighborhood of $I \times \{0\}$ since $f$, as a $\GL_n$ bundle map, is independent of scale. So we are free to choose $\phi$ a bit differently:

Given $\kkk > 0$, let $\phi: [0, \infty) \to [0, 1]$ be a smooth, monotonically non-decreasing map such that $\phi([0, 1]) = 0$, $\phi(s) = 1$ outside some closed interval, and, for all $s$, $s \phi'(s) < \kkk$.  For example, $\phi$ could be obtained by integrating a smooth approximation to the discontinuous function on $[0, \infty)$ which takes the value $\kappa/2s$ for $s \in [1, e^{\frac2\kkk}]$ but is otherwise $0$.

Suppose $A$ is any upper triangular matrix with positive entries in the diagonal and $\tau \in [0, 1]$.  Then the matrix $A_\tau = (I + \tau(A-I)) = (1 - \tau)I + \tau A$ is invertible, since it also is upper triangular and has positive diagonal entries.  Suppose $\kkk_1, \kkk_2 \in [0, \infty)$ satisfy $$\kkk_1 \leq \frac1{sup_\tau |A_\tau^{-1}|}, \hspace{1in} \kkk_2 \leq \frac1{|A - I|}.$$  Then, for any $\tau\in [0, 1]$ and any $y, z \neq 0 \in \mbbR^n$, $$|A_\tau z| \geq \kkk_1|z|, \hspace{1in} \kkk_2|(A - I)y| \leq |y|.$$  In particular, if $0 \leq \kkk < \kkk_1 \kkk_2$ then $$|A_\tau(z) + \kkk\frac{|z|}{|y|}(A-I)y| \geq |A_\tau(z) | - \kkk\frac{|z|}{|y|}|(A-I)y| > \kkk_1|z| -  \kkk_1\frac{|z|}{|y|}|y| = 0.$$

Apply this to the problem at hand by choosing values $\kkk_1, \kkk_2$ so that for all $x \in [0, 1]$, $$\kkk_1 \leq \frac1{sup_\tau |(T_x)_\tau^{-1}|}, \hspace{1in} \kkk_2 \leq \frac1{|T_x - I_n|}.$$  Then choose $\phi: [0, \infty) \to [0, 1]$ as above so that for all $s \in  [0, \infty)$, $s \phi'(s) < \kkk_1 \kkk_2$.  Much as in Step 2, define the smooth homotopies $g_t: \mbbR^n \to [0, 1]$ and $f_t|\mbbR^n_x: \mbbR^n_x \to \mbbR^n_x$ via 
\begin{align*} g_t(y) &= (1-t) + t\phi(|y|)(y) \\ f_t(y) &= (I_n + g_t(y)(T_x - I_n)).\end{align*}   
Then for all $x \in I$ and each $y \neq 0 \neq z \in \mbbR^n_x$, 
\begin{align*}D(f_t|\mbbR^n_x)_y(z) &= (I + g_t(y)(T_x - I_n))z + Dg_t(y)(z)(T_x - I_n)y \\ &= (I + g(y)(T_x - I_n))z +  t\phi'(|y|)\frac{y \cdot z}{|y|}(T_x - I_n)y.\end{align*}
Now notice that by construction $$t\phi'(|y|)\frac{y \cdot z}{|y|} \leq \phi'(|y|)|y|\frac{y \cdot z}{|y|^2} < \kkk_1\kkk_2 \frac{|z|}{|y|}$$ so, applying the argument above to $\tau = g(y)$ and $A = T_x$, we have $|D(f_t|\mbbR^n_x)_y(z)|>0$.  Thus $D(f_t|\mbbR^n_x)$ is non-singular everywhere and so, for all $t$, $f_t$ is a diffeomorphism, as required.  

The extension to a parameterized family $f^u, u \in B^k$ is essentially the same as in the previous two cases, after choosing $\kappa$ to be less than the infimum of $\kappa_1\kappa_2$ taken over all $u \in B^k$.  Note that for $u$ near $\bdd B^k$, where $f^u$ is already an orthogonal bundle map, each $T_x$ will be the identity, so, regardless of the value of $t$ or $y$ in the definition $(f_t|\mbbR^n_x)y = (I_n + g_t(y)(T_x - I_n))y,$ the function is constantly the identity.  
\end{proof}

\bigskip

\noindent {\bf Stage 6: From preserving $H'$ to preserving $H_2$.}

\medskip

The previous stages allow us to define a possibly very thin regular neighborhood $H' \subset H_2$ of $\bdd H \cup I_0$ and a relative homotopy of $\Theta$ to a map $(B^k, \bdd B^k) \to (\Diff(H, H', I_0), \Diff(H, H_2, I_0))$.   Continue to denote the result as $\Theta$ and invoke this special case of Hatcher's powerful theorem:

\begin{lemma} \label{lemma:pseudo} \label{lemma:Hatch} Suppose $F$ is a closed orientable surface.  Any $\psi: S^k \to \Diff(F \times I, F \times \{ 0 \} )$ is homotopic to a map so that each $\psi(u): F \times I \to F \times I$ respects projection to $I$. (In fact, unless $F$ is a torus or a sphere, there is a diffeomorphism $f: F \to F$ and a homotopy of $\psi$ to a map so that each $\psi(u)$ is just $f \times id_I$.)
\end{lemma}

\begin{proof}  The special case in which $F$ is a torus or a sphere is left to the reader; it will not be used.  Pick a base point $u_0 \in S^k$ and let $\psi_0 = \psi(u_0): F \times I \to F \times I$.  Take $f$ in the statement of the Lemma to be $\psi_0|(F \times \{ 0 \}): F \to F$ and, with no loss of generality, assume this diffeomorphism is the identity.  Since $\Diff(F)$ is contractible \cite{EE}, the map $\psi_|: S^k \to \Diff(F)$ defined by $\psi_|(u) = \psi(u)|(F \times \{ 0 \})$ can be deformed so that each $\psi_|(u)$ is the identity.  The homotopy of $\psi_|$ induces a homotopy of $\psi$ via the isotopy extension theorem.  

The map $p_1\psi_0: F \times I \to F$ defines a homotopy from $\psi_0|(F \times \{1\}):F \to F$ to the identity, and this implies that  $\psi_0|(F \times \{1\})$ is isotopic to the identity.  The previous argument applied to diffeomorphisms of 
$F \times \{1\}$ instead of $F \times \{0\}$ then provides a further homotopy of $\psi$, after which each diffeomorphism $\psi(u): F \times I \to F \times I$ is the identity on $F \times \{ 0, 1 \} = \bdd(F \times I)$.  That is, after the homotopy of $\psi$, $\psi$ maps $S^k$ enirely into the space of diffeomorphisms of $F \times I$ that are the identity on $\bdd (F \times I)$.  The lemma then follows from the central theorem of \cite{Ha1}.
\end{proof}

The region $H_2 - \inter(H')$ between the regular neighborhoods of $\bdd H \cup \I_0$ is a collar which can be parameterized by $\Sss \times I$.  With this in mind, apply Lemma \ref{lemma:pseudo} to $\psi = \Theta|\bdd B^k: S^{k-1} \to \Diff(H_2 - \inter(H'))$, extending the parameterization slightly outside of $H_2 - \inter(H')$ via an argument like that in Stage 2 above.  We then have a parameterization $\Sss \times \mbbR$ of a neighborhood of $H_2 - \inter(H')$ so that $\bdd H_2$ corresponds to $\Sss \times \{1\}$, $\bdd H'$ corresponds to $\Sss \times \{0\}$, and so that for each $u \in \bdd B^k$, the restriction of $\Theta(u): H \to H$ to $\Sss \times \mbbR$ respects projection to $\mbbR.$   Let $g_t: \mbbR \to \mbbR$ be a smooth isotopy with compact support from the identity to a diffeomorphism that takes $1 \in \mbbR$ to $0.$  Define the isotopy $r_t: \Sss \times \mbbR \to \Sss \times \mbbR$ by $r_t(y, s) = (y, g_t(s))$ and extend by the identity to an isotopy of $H$.  The result is an isotopy from the identity to a diffeomorphism that takes $H_2$ to $H'$.  Then the deformation $\Theta_t$ of $\Theta$ defined by $\Theta_t(u) = r_t^{-1}\Theta(u)r_t: H \to H$ pairwise homotopes $\Theta: (B^k, \bdd B^k) \to (\Diff(H, I_0), \Diff(H, H_2, I_0))$  to a map whose image lies entirely in $\Diff(H, H_2, I_0)$.   

\end{document}